\documentclass[11pt]{amsart}

\usepackage{amssymb,amsthm,amsmath}

\theoremstyle{plain}
\newtheorem{prop}{Proposition}[section]
\newtheorem{lem}[prop]{Lemma}

\newtheorem{thm}[prop]{Theorem}

\theoremstyle{definition}

\newtheorem{mydef}[prop]{Definition}

\newtheorem{conj}[prop]{Conjecture}

\theoremstyle{remark}
\newtheorem{remark}[prop]{Remark}

\DeclareMathOperator{\supp}{supp}

\newcommand\N{\mathbb{N}}

\newcommand\R{\mathbb{R}}

\begin{document}
\title{Exceptional sets in homogeneous spaces and Hausdorff dimension.}
\author{Shirali\@ Kadyrov}

\address{Department of Mathematics,
Nazarbayev University,
Astana, Kazakhstan.}
\email{shirali.kadyrov@nu.edu.kz}
\keywords{Exponential mixing, Homogeneous dynamics, Hausdorff dimension, Open dynamics}
\subjclass[2010]{Primary: 37A17, Secondary: 11K55, 37A25}

\begin{abstract}
In this paper we study the dimension of a family of sets arising in open dynamics. We use exponential mixing results for diagonalizable flows in compact homogeneous spaces $X$ to show that the Hausdorff dimension of set of points that lie on trajectories missing a particular open ball of radius $r$ is at most 
$$\dim X + C\frac{r^{\dim X}}{\log r},$$
where $C>0$ is a constant independent of $r>0$. Meanwhile, we also describe a general method for computing the least cardinality of open covers of dynamical sets using volume estimates.

\end{abstract}

\maketitle

\section{Introduction}
As a subbranch of dynamical systems, open dynamical systems is an active research area. A typical example is the dynamics of a billiard ball on a table with holes. For some of the recent developments in open dynamics research we refer to \cite{BDM10, BY11,DWY10,FP12, Kel12, KL09}. In this article, we would like to study the Hausdorff dimension of the open systems with holes arising in homogeneous spaces. Let $G$ be a Lie group and $\Gamma$ a uniform lattice in $G$. Consider the compact homogeneous space $X=G/\Gamma$ and one parameter semigroup $(g_t)_{t \ge 0}$ of diagonalizable elements of $G$ acting on $X$ by left translations. In many situations, using uniqueness of the measure of maximal entropy and variational principle for topological entropy one can show that the Hausdorff dimension of the set of trajectories under $g_t$ that miss a fixed open set is strictly less than $\dim G$. Our goal in this paper is to find an effective bound for the Hausdorff dimension of such sets. We fix a right invariant Riemannian metric $d$ on $G$ and use the same notation $d$ for the induced metric on $X$. Let $H$ be the unstable subgroup of $G$ w.r.t. $g_1$, that is,
$$H:=\{g \in G : d(g_n g g_{-n},e) \to 0 \text{ as } n\to -\infty\},$$
where $e \in G$ is the identity element. Let $\nu$ be the Lebesgue measure on $H$, $\mu$ the probability Haar measure on $X$, $\|\cdot\|_\ell$  a Sobolev norm in $W_\ell^2(X)$, and $\| \cdot\|$ is the maximum norm.

\begin{mydef}\label{prop:EM} We say that a flow $(X,\{g_t\})$ has \emph{property (EM)} if
there exist $D,\lambda, k,\ell>0$ such that for any $x \in X$, $f \in C_c^\infty (H)$, and $\psi \in C^\infty(X)$ such that the map $g \to g x$ is injective on some ball in $G$ containing $\supp(f)$, and for any $t \ge 0$ one has
$$\left|\int_Hf(h) \psi(g_t h x) d\nu(h)-\int_H fd\nu(h) \int_X\psi d\mu \right| \le {\rm const}(f,\psi) e^{-\lambda t},$$
where
$$ {\rm const}(f,\psi)=D \|f\|_\ell \|\psi\|_\ell \left(\max_{x \in X}\|\nabla \psi\| \int_H |f| d\nu \right)^k.$$
\end{mydef}

Our definition of property $(EM)$ is a consequence, (c.f. \cite[Proposition~2.4.8]{KleMar}), of property $(EM)$ defined in \cite{KleMar}. We note that the $\text{const}(f,\psi)$ is not stated explicitly in \cite{KleMar}, however as pointed out in \cite{BK13}, one easily gets the above explicit constant.

From \cite[Lemma~2.4.1]{KleMar} we see that, in particular, if $\Gamma$ is an irreducible lattice in $G$ and $G$ has an essential factor $G'$ which is not locally homeomorphic to $SO(m,1)$ or $SU(m,1), m \in \N$ then the property $(EM)$ holds for the flow $(X, \{g_t\})$.
 
For any metric space $(Y,d)$ we define open balls by
$$
B^Y(x,r):=\{y \in Y : d(x,y)<r\}.$$
When $Y$ is a group with identity $e$ we simply let $B^Y(r):=B^Y(e,r).$ 
Since $X$ is compact we can pick a constant $r_0\in (0,1)$ such that for any $x \in X$ the map $g \to g x$ is injective on $B^G(r_0)$. We consider $\Gamma=e \Gamma$ as an identity coset in $X$ and for simplicity write $B^X(r)$ for $B^X(\Gamma,r)$. Our main result is the following.

\begin{thm}\label{thm:main'} Let $G,\Gamma,g_t,\mu,r_0$ be as above. Assume that the flow $(X, \{g_t\})$ has property (EM). Then, there exists a constant $C>0$ such that for any $r\in (0,r_0)$ and $x_0 \in X$ the set
\begin{equation}\label{eqn:exceptionalset}
\{x \in X : g_t x\not \in B^X(x_0,r) \text{ for any } t \ge 0 \}
\end{equation}
has Hausdorff dimension at most 
$$\dim G + C \frac{\mu(B^X( r))}{\log \mu(B^X(r))}.$$
\end{thm}

Similar results were studied in different contexts, see e.g. \cite{FP12, KL09, BK13, Kad13, Urb87,Hen92}. It is desired to obtain the similar results when $\Gamma$ is not necessarily uniform. However, to derive nontrivial estimates from exponential mixing we need to assume that the complement of an open ball in $X$ is compact which is not the case if $X$ is not compact. We will make no claim on sharpness of our result. In fact, we think the following should hold true.

\begin{conj}
Let $G$ be a Lie group, $\Gamma$ a lattice, $\{g_t\}$ a one parameter diagonalizable subgroup of $G$ acting on $X=G/\Gamma.$ Assume that the Haar measure $\mu$ on $X$ is the unique measure of maximal entropy. Then, there exists a constant $C>0$ such that for any $x_0 \in X$ and $r>0$ such that the map $g \to g x_0$ is injective on $B^G(r)$, the set \eqref{eqn:exceptionalset} has Hausdorff dimension at most $\dim G - C  \mu(B^X(r)).$
\end{conj}

In this generality, it is not known if the Hausdorff dimension of the set \eqref{eqn:exceptionalset} is strictly less than $\dim G$. This was observed in \cite{EKP13} for $X=G/\Gamma$ when $G$ is of $\R$-rank one.

We will adopt some of the ideas from \cite{BK13} where the Hausdorff dimensions of the level sets of uniformly badly approximable systems of linear forms were estimated. In the next section we will prove some known results in a more general setting. In \S~\ref{sec:EM} we use the property $(EM)$ to obtain Theorem~\ref{thm:main'}.

\subsection*{Acknowledgement}
The author would like to thank the referee for useful comments.

\section{Auxiliary results}\label{sec:auxiliary}

This section may be of independent interest. In certain cases, it may be possible to estimate the measure of certain subsets arising in ergodic theory and one may wish to study the covering of these sets by certain balls. In this section we provide a general version of some well known results useful in such occasions. See e.g. \cite{KleMar} where tessellation method was introduced for such purposes. 

Let $G$ be a Lie group and $\Gamma$ a discrete subgroup of $G$ and set $X=G/\Gamma$. Let $\{g_t\}$ be a one parameter diagonalizable semigroup and $H$ be the unstable subgroup w.r.t. $g_1$. As before, let us fix a right invariant Riemannian metric $d$ on $G$ and to simplify the notation let $B(r):=B^H(r)$. For any set $K$ of $X$ we let $\partial_r K$ denote the thickening of $K$ by $r$, that is, $\partial_r K=\{x \in X: d(x,K)\le r\}$. 
Let $\nu$ be the Haar measure on $H$. By Bowen $(t,r)$-ball in $H$ we mean any translate of $g_{-t} B(r) g_t$ in $H$.

\begin{prop}\label{prop:exceptionalk}
For any $r,t \ge 0$, a subset $K$ of $X$, a point $x \in K$, and a natural number $k \in \N,$ let $E_k(t,r,K,x)$ denote the set given by
\begin{equation}\label{eqn:exeptionalk}
E_k(t,r,K,x):=\{h \in B(r): g_{t\ell} h x  \in K, \ell=1,2,\dots,k\}.
\end{equation} 
We have that the set $E_k(t,r,K,x)$ can be covered with at most
$$\sup_{x' \in \partial_r K} \left(\frac{\nu(E_1(t,2r,\partial_r K, x'))}{\nu(g_{-t}B(r/2) g_t)}\right)^k$$
Bowen $(tk,r)$-balls in $H$.
\end{prop}
 
Proposition~\ref{prop:exceptionalk} follows inductively from the following.

\begin{lem}\label{lem:ratio}
For any $r,t \ge 0$, a point $x \in X$ and a subset $K$ of $X$ we have that the set $E_1(t,r,K,x)$ can be covered with 
$$\frac{\nu(E_1(t,2r,\partial_r K, x))}{\nu(g_{-t}B(r/2) g_t)}$$
Bowen $(t,r)$-balls in $H$.
\end{lem}

\begin{proof}
Let us pick $h \in E_1(t,r, K, x)$ so that $g_t h x \in K$. Any $h' \in   g_{-t}B(r) g_t h$ satisfies
$$
d(g_t h' x, K) \le d(g_t h' x,g_t h x) \le d(g_t h',g_t h) <r.
$$
Hence, $g_t h' x  \in \partial_r K$. Moreover, $g_{-t}B(r) g_t B(r) \subset B(2r)$ so that $ g_{-t}B(r) g_t h  \subset E_1(t,2r,\partial_rK,x)$ whenever $h \in E_1(t,r,K,x)$. A finite subset $S$ of $E_1(t,r,K,x)$ is said to be $(t,r)$-\emph{separated} if for any distinct $h,g \in S$ we have $d(g_t h,g_t g) \ge r$. Let $S$ be a $(t,r)$-separated set of $E_1(t,r,K,x)$ with maximum cardinality. Then, it suffices to estimate $\# S$ from above as maximality implies that $E_1(t,r,K,x) \subset \cup_{h \in S}g_{-t}B(r) g_th$. On the other hand, the sets $g_{-t}B(r/2) g_t h$ and $g_{-t}B(r/2) g_tg$ are disjoint for distinct $g, h \in S$ since otherwise if $g_{-t}h' g_t h=g_{-t}g' g_tg$ with $h',g' \in B(r/2)$ then we would have
\begin{multline*}
d(g_t g,g_t h) =d(g_t g, (h')^{-1}g' g_t g)=d(e, (h')^{-1}g')\\
\le d(e,g')+d(g',(h')^{-1}g') <r.
\end{multline*}
Also, any $h \in S$ satisfies $g_{-t}B(r/2) g_t h \subset E_1(t,2r,\partial_rK,x)$ so that 
$$\bigsqcup_{h \in S}g_{-t}B(r/2) g_th \subset  E_1(t,2r,\partial_rK,x).$$
 Thus,
$$\# S \le \frac{\nu(E_1(t,2r,\partial_r K, x))}{\nu(g_{-t}B(r/2) g_t)}. \qedhere$$
\end{proof}

\begin{remark}
We note that in the conclusion of Lemma~\ref{lem:ratio} and Proposition~\ref{prop:exceptionalk} one may replace $2r$ by $r+\frac12{ \rm diam}( g_{-t}B(r/2) g_t )$.
\end{remark}

\begin{proof}[Proof of Proposition~\ref{prop:exceptionalk}]
The case $k=1$ is considered in Lemma~\ref{lem:ratio}. Assume that $k \ge 2$. We note that $E_{k}(t,r,K,x)\subset E_{k-1}(t,r,K,x)$ so that any covering of $E_{k-1}(t,r,K,x)$ also provides a covering for $E_{k}(t,r,K,x).$  Assume that the set $E_{k-1}(t,r,K,x)$ can be covered with 
\begin{equation}\label{eqn:inductive}
\sup_{x' \in \partial_rK}\left(\frac{\nu(E_1(t,2r,\partial_r K, x'))}{\nu(g_{-t}B(r/2) g_t)}\right)^{k-1}
\end{equation}
 Bowen $(t(k-1),r)$-balls in $H$ and let $g_{-t(k-1)}B(r) g_{t(k-1)} h$ be one of these Bowen balls. We are interested in the covering of the set  
$$D:=g_{-t(k-1)}B(r) g_{t(k-1)} h \cap E_k(t,r,K,x).$$
If the set $D$ is empty there is nothing to cover. Let us assume that $D$ is nonempty. We claim that $D':=g_{t(k-1)}D h^{-1} g_{-t(k-1)} \subset E_1(t,r,K,x'')$ where $x''=g_{t(k-1)}hx$. First note that $D' \subset B(r).$ Next, for any $h' \in D'$ we note that $g_{-t(k-1)} h' g_{t(k-1)} h \in D$ so that 
$$g_{tk} g_{-t(k-1)} h' g_{t(k-1)} hx=g_t h' x'' \in K.$$
This proves the claim. Now, we apply Lemma~\ref{lem:ratio} to see that the set $D'$ can be covered with at most
\begin{equation}\label{eqn:x''}
\frac{\nu(E_1(t,2r,\partial_r K, x''))}{\nu(g_{-t}B(r/2) g_t)}
\end{equation}
Bowen $(t,r)$-balls. Fix any $h'' =g_{-t(k-1)} g g_{t(k-1)} h \in D$, then 
$$d(x'', K) \le d(g_{t(k-1)}hx, g_{t(k-1)}h''x ) \le d(g_{t(k-1)}h, g g_{t(k-1)} h ) \le d(e,g) <r,$$
where $e \in G$ is the identity element. So, $x'' \in \partial_r K.$ Together with  \eqref{eqn:x''}
 we see that the original set $D$ can be covered with
\begin{equation}
\frac{\nu(E_1(t,2r,\partial_r K, x''))}{\nu(g_{-t}B(r/2) g_t)} \le \sup_{x' \in \partial_r K} \frac{\nu(E_1(t,2r,\partial_r K, x'))}{\nu(g_{-t}B(r/2) g_t)}
\end{equation}
Bowen $(tk,r)$-balls. Finally, combining with \eqref{eqn:inductive} we obtain that the set $E_k(t,r,K,x)$ can be covered with 
\begin{equation*}
\sup_{x' \in \partial_rK}\left(\frac{\nu(E_1(t,2r,\partial_r K, x'))}{\nu(g_{-t}B(r/2) g_t)}\right)^{k}
\end{equation*}
 Bowen $(tk,r)$-balls in $H$.
\end{proof}

\section{Proof of Theorem~\ref{thm:main'}}\label{sec:EM}

In this section we use the exponential mixing estimates, that is, the property $(EM)$ to deduce Theorem~\ref{thm:main'}.  For any $r>0$ and $x_0 \in X$ we define the compact set 
$$K(x_0,r)=\{x \in X : d(x,x_0) \ge r\}.$$
We recall that $\nu$ is the Lebesgue measure on $H$ and $B(r)=B^H(e,r)$.
Using the countable stability of the Hausdorff dimension it is easy to see that Theorem~\ref{thm:main'} follows from the following.

\begin{thm}\label{thm:main} Assume that the flow $(X,\{g_t\})$ has property $(EM)$. Then, there exists a constant $C>0$ such that for any $x,x_0 \in X$ and $r>0$ sufficiently small so that $g \to g x_0$ is injective on $B^G(r)$, the set
$$E(r/2,K(x_0,r), x):=\{h \in B(r/2): g_t h x \in K(x_0,r), \forall t \ge 0\}$$
has Hausdorff dimension at most $\dim H + C \frac{ \mu(B^X(r))}{\log \mu(B^X(r))}.$
\end{thm}

We now give the proof of Theorem~\ref{thm:main'} assuming Theorem~\ref{thm:main}.

\begin{proof}[Proof of Theorem~\ref{thm:main'}]
Let $x_0\in X$ and $r>0$ be given. From countable stability of the Hausdorff dimension it suffices to show that for any $x \in X$ there exists some $r'>0$ such that the set
$$ \{g \in B^G(r')  : g_t g x \in K(x_0,r) \text{ for any } t\ge 0\}$$
has Hausdorff dimension at most $\dim X + C \frac{ \mu(B^X(r))}{\log\mu(B^X(r))}.$ Let $H^0=\{g \in G : g g_1=g_1g\}$ and $H^-=\{g \in G : g_n g g_{-n} \to e \text{ as } n \to \infty\}$. Then, locally $G$ is a direct product of $H^-,H^0$ and $H$. Once $r'>0$ is sufficiently small, for any $g \in B^G(r')$ we may write $g=h'h$ where $h \in B(r')$ and $h' \in B^{H^-H^0}(r')$. Note that for any $y \in X$ we have 
$$d(g_t g x, y) \le d(g_t h'hx,g_t h x)+d(g_t h x, y) \le d(g_t h x,y)+r'.$$ 
Thus, $g_t gx \in K(x_0,r)$ implies $g_t h x \in \partial_{r'}K(x_0,r) \subset K(x_0,r-r')$ for any $r' \in (0,r)$. So, we deduce that the Hausdorff dimension of the set \eqref{eqn:exceptionalset} is at most
$$ \dim_H \left(\{h \in B^H(r): g_t h x \in K(x_0,r-r'), \forall t\ge 0\} \times H^-  H^0\right)
$$
for any $r'\in (0,r)$. By letting, $r' \to 0$ and using $\dim X=\dim G=\dim H+ \dim (H^-H^0)$ we see that the Hausdorff dimension of the set \eqref{eqn:exceptionalset} is at most
$$\dim X + C \frac{ \mu(B^X(r))}{\log\mu(B^X(r))},$$
as desired.
\end{proof}

For the remaining of the article we focus our attention to the proof of Theorem~\ref{thm:main}. To estimate the Hausdorff dimension we need to find covering of the set with small balls. Clearly, $E(r/2,K(x_0,r),x) \subset E_k(t,r/2,K(x_0,r),x)$ for any $k \in \N$ and $t \ge 0$, where the latter were defined in the previous section. To apply the results obtained in the previous section we need to obtain a measure estimate of 
$$E_1(t,r,\partial_{r/2} K(x_0,r),x)=E_1(t,r, K(x_0,r/2),x).$$
 To optimize the estimates, we will later specify $t \ge 0$.

Let $n=\dim H$ and $m=\dim G$. Recall that $r_0>0$ is an injectivity radius of $X$. From property $(EM)$ we will deduce the following.

\begin{prop}\label{prop:measureEst}
 Let $(X,\{g_t\})$ be the flow with property $(EM)$.
There exist constants $D,E,\lambda'>0$ such that for any $x, x_0 \in X$, $r\in (0,r_0/2)$, and $t\ge \frac{1}{\lambda'}\log \frac1r$ we have
$$\nu(E_1(t,r, K(x_0,r/2),x)) \le \nu(B(r)) -  D r^{m+n}+ E e^{-\lambda' t}.$$
\end{prop}

To make use of the (EM) property we need smooth functions that approximates the characteristic functions of small balls.

\begin{lem}\label{lem:feps}
For any unimodular Lie group $G'$ of dimension $d$ and $\ell \in \N$ there exist constants $M_{\ell,d},M_{1,d}>0$ such that the following holds. For any $\epsilon, r>0$, there exist functions $f_\epsilon:G' \to [0,1]$ such that 
\begin{itemize} 
\item $f_\epsilon \equiv 1$ on $B^{G'}(r)$
\item  $f_\epsilon\equiv 0$ on $\left(B^{G'}(r+\epsilon)\right)^c$
\item $\|f_\epsilon\|_\ell \le M_{\ell,d} \epsilon^{-(d+\ell)}$
\item $\|\nabla f_\epsilon\|  \le M_{1,d} \epsilon^{-1-d}.$
 \end{itemize}
\end{lem}

\begin{proof}
The proof is essentially contained in \cite{BK13} and for completeness we will recall the proof here. Let $g:B^{G'}(1) \to [0,\infty]$ be a smooth function with $\|g\|_{L^1}=1$ and define $g_\epsilon(x)=(c(d)/\epsilon^d)g(2x/\epsilon)$ where $c(d)>0$ is chosen so that $\|g_\epsilon\|_{L^1}=1.$ For any $\alpha=(\alpha_1,\alpha_2,\dots,\alpha_d)$ let $D_\alpha$ denote the differential operator $\frac{\partial^{|\alpha|}}{\partial x_1^{\alpha_1}\dots \partial x_d^{\alpha_d}}$. The convolution $f_\epsilon=g_\epsilon * 1_{B^{G'}(r+\epsilon/2)}$ is smooth as $D_\alpha f_\epsilon=(D_\alpha g_\epsilon) * 1_{B^{G'}(r+\epsilon/2)}$ and it has the support contained in 
$$\supp(g_\epsilon) \supp(1_{B^{G'}(r+\epsilon/2)}) \subset B^{G'}(r+\epsilon).$$
Clearly, $0 \le f_\epsilon $ as $0 \le g_\epsilon$. From $\int g_\epsilon =1$ it follows that $f_\epsilon \le 1$ and further with unimodularity assumption we see that $f \equiv 1$ on $B^{G'}(r)$. Thus, it remains to prove the last two assertions. As $g_\epsilon$ is smooth and compactly supported we see that $|D_\alpha f_\epsilon(x)| \le M_{|\alpha|,d} \epsilon^{-d-|\alpha|}$. This clearly implies $\|\nabla f_\epsilon\|  \le M_{1,d} \epsilon^{-1-d}$ and Young's Inequality yields
$$\|D_\alpha f_\epsilon\|_{L^2} \le \|D_\alpha g_\epsilon\|_{L^2} \|1_{B^{G'}(r+\epsilon/2)}\|_{L^1}  \le M_{|\alpha|,d} \epsilon^{-d-|\alpha|},$$
so that $\|f_\epsilon\|_\ell \le M_{\ell,d} \epsilon^{-(d+\ell)}$.
\end{proof}

\begin{proof}[Proof of Proposition~\ref{prop:measureEst}]

Let us fix some $r\in(0,r_0/2)$, $x_0 \in X$ and define 
$$A(t,x)=\{h \in B(r) : g_t h x\in B^X(x_0,r/2)\}.$$
Observe that $E_1(t,r, K(x_0,r/2),x) \sqcup A(t,x) =B(r)$. Thus, it suffices to prove 
$$\nu(A(t,x)) \ge D r^{m+n}- E e^{-\lambda' t}.$$
We choose $\lambda'>0$ such that $\lambda - (2\ell+m+n+k+km)\lambda' > \lambda'$ and for any $t \ge \frac{1}{\lambda'} \log \frac1r$ we consider the following functions $f_t,\psi_t$. We apply Lemma~\ref{lem:feps} for $G'=H, d=n, r>0$ and $\epsilon=e^{-\lambda' t}$ to obtain $f_t:=f_{e^{-\lambda' t}}$. Then apply Lemma~\ref{lem:feps} for $G'=G, d=m, r/2>0$ and $\epsilon=e^{-\lambda' t}$ to obtain $\psi_t.$ We will identify $\psi_t$ with the function on $X$ given by $g x_0 \to \psi_t(g)$. This is well defined as $g \to g x_0$ is injective on $B^G(2r)$ and $\supp(\psi_t) \subset B^G(2r)$ whenever $t \ge \frac{1}{\lambda'} \log \frac1r$.  From property $(EM)$ we have
\begin{multline*}
\left|\int_Hf_t(h) \psi_t(g_t h x) d\nu(h)-\int_H f_td\nu \int_X\psi_t d\mu \right| \\
\le D\|f_t\|_\ell \|\psi_t\|_\ell \left(\max_{x \in X}\|\nabla \psi_t(x)\| \int_H |f_t| d\nu \right)^k e^{-\lambda t}.
\end{multline*}
So,
\begin{multline*}
\int_Hf_t(h) \psi_t(g_t h x) d\nu(h) \ge \int_H f_td\nu(h) \int_X\psi_t d\mu - D_1 e^{\lambda' t(2\ell+m+n+k+km)-\lambda t}\\
\ge  \int_H f_td\nu \int_X\psi_t d\mu - D_1 e^{-\lambda' t}.
\end{multline*}
Let $f=1_{B(r)}$ and $\psi=1_{B^X(x_0,r/2)}$. Then,
\begin{align*}
\nu(A(t,x)) &=\int_H f(h) \psi(g_t h x) d\nu(h)\\
& \ge \int_H f_t(h) \psi(g_t h x) d\nu(h)-\int |f_t-f| d\nu\\
& \ge \int_H f_t(h) \psi(g_t h x) d\nu(h)-\nu\left(B(r+e^{-\lambda' t})\backslash B(r)\right)\\
& \ge \int_H f_t(h) \psi(g_t h x) d\nu(h)-C_1 e^{-\lambda' t}.
\end{align*}
Similarly using $\psi \ge \psi_t - |\psi_t-\psi|$ we obtain
\begin{align*}
\nu(A(t,x)) & \ge \int_H f_t(h) \psi_t(g_t h x) d\nu(h)-2C_1 e^{-\lambda' t}\\
& \ge  \int_H f_td\nu \int_X\psi_t d\mu - D_1 e^{-\lambda' t}-2C_1 e^{-\lambda' t}\\
& \ge D r^{n+m}- (D_1+2C_1) e^{-\lambda' t}.
\end{align*}
Taking $E=D_1+2C_1$  we obtain the proposition.
\end{proof}

\begin{proof}[Proof of Theorem~\ref{thm:main}]
From Proposition~\ref{prop:exceptionalk} we see that for any $k \in \N$ the set $E_k(t,r/2,K(x_0,r),x)$ (and hence the set $E(r/2,K(x_0,r),x)$) can be covered with at most
\begin{equation}
\sup_{x' \in X} \left(\frac{\nu(E_1(t,r, \partial_{r/2} K(x_0,r), x'))}{\nu(g_{-t}B(r/4) g_t)}\right)^k
\end{equation}
Bowen $(tk,r/2)$-balls. Since $\partial_{r/2} K(x_0,r) = K(x_0,r/2)$, using Proposition~\ref{prop:measureEst} we get that
$$\nu(E_1(t,r, \partial_{r/2} K(x_0,r),x')) \le \nu(B(r)) -  D r^{m+n}+ E e^{-\lambda' t} \text{ for any } x' \in X.$$
We define 
$$\lambda_0=\max \{|\lambda| : \lambda \text{ is an eigenvalue of } {\rm Ad}_{g_1}\}.$$ 
Then it is easy to see that any Bowen $(tk,r/2)$-ball can be covered with 
$$ O\left(  \frac{\nu(g_{-tk} B(r/2) g_{tk})}{\nu(B({r e^{-\lambda_0 t k}}))}\right)=O\left(  \frac{\nu(g_{-t} B(r/4) g_{t})^k}{\nu(B({r e^{-\lambda_0 t }}))^k}\right)$$
balls of radius $r e^{-\lambda_0 t k}$, where the implicit constant does not depend on $t,r,k$.
Hence, the set $E(r/2,K(x_0,r),x)$ can be covered with
$$ O\left(\left(\frac{\nu(B(r)) -  D r^{m+n}+ E e^{-\lambda' t}}{\nu(B({r e^{-\lambda_0 t }}))}\right)^k\right)$$ 
balls of radius $re^{-\lambda_0tk }$. Note that $\nu(B({r e^{-\lambda_0 t}}))=\nu(B(r)) e^{-\lambda_0 n t} =O( r^n e^{-\lambda_0 n t})$. Let $\underline\dim_B$ and $\dim_H$ denote the lower box dimension and the Hausdorff dimension respectively. Then,
\begin{align*}
\underline\dim_B E(r/2,K(x_0,r),x) & \le \lim_{k \to \infty} \frac{k \log \left(\frac{\nu(B(r)) -  D r^{m+n}+ E e^{-\lambda' t}}{\nu(B({r e^{-\lambda_0 t }}))}\right)}{- \log (re^{-\lambda_0t k})} \\
&=\frac{\log \left(e^{\lambda_0 n t} \left(1- D' r^m+E' r^{-n} e^{-\lambda' t} \right)\right)}{\lambda_0 t}\\
&=n+ \frac{\log  \left(1- D' r^m+E' r^{-n} e^{-\lambda' t} \right)}{\lambda_0 t}
\end{align*}
To optimize the estimate we want to choose a suitable $t \ge \frac{1}{\lambda'} \log \frac1r$. We have $r \in (0,r_0)$ and $r_0<1$. Let us pick $p \ge 1$ such that $r^p <r_0^p \le \frac{D'}{2E'}$ and let $t=\frac{m+n+p}{\lambda'}\log \frac1r$. Then, for any $r\in(0,r_0)$ we have
$$E'r^{-n}e^{-\lambda't} = E' r^{-n} r^{m+n+p} < \frac{D'}2 r^m. $$
Thus,
\begin{multline*}\dim_H  E(r/2,K(x_0,r),x) \\
\le \underline\dim_B E(r/2,K(x_0,r),x)  \le n + \frac{\log  \left(1- \frac{D'}2 r^m \right)}{\lambda_0 t} \le n - D'' \frac{ r^m}{ \log \frac1r}. 
\end{multline*}
This finishes the proof.
\end{proof}

\end{document}